\documentclass[a4paper,english]{amsart}
\usepackage[latin1]{inputenc}
\usepackage{mathtools}
\usepackage{amsthm}
\usepackage{amssymb}
\usepackage{mathrsfs}
\usepackage{a4wide}
\usepackage{babel}
\usepackage{color}

\numberwithin{equation}{section}
\providecommand{\definitionname}{Definition}
\providecommand{\remarkname}{Remark}
\providecommand{\theoremname}{Theorem}
\newtheorem{thm}{\protect\theoremname}
\newtheorem{defn}[thm]{\protect\definitionname}
\newtheorem{rem}[thm]{\protect\remarkname}

\allowdisplaybreaks
\newcommand{\R}{\mathbb{R}}

\newcommand{\dom}{\operatorname{dom}}
\newcommand{\ran}{\operatorname{ran}}
\newcommand{\grad}{\operatorname{grad}}
\newcommand{\curl}{\operatorname{curl}}
\newcommand{\id}{\operatorname{id}}
\newcommand{\mathS}{\mathscr{S}}
\newcommand{\tr}{\operatorname{tr}}
\newcommand{\skw}{\operatorname{skw}}
\newcommand{\sym}{\operatorname{sym}}
\newcommand{\dev}{\operatorname{dev}}

\newcommand{\dv}{\operatorname{div}}
\newcommand{\Grad}{\operatorname{Grad}}
\newcommand{\Curl}{\operatorname{Curl}}
\newcommand{\Div}{\operatorname{Div}}
\newcommand{\ed}{\operatorname{d}}
\newcommand{\cd}{\operatorname{\delta}}
\newcommand{\eps}{\varepsilon}
\def\papia{a}
\def\S{\mathbb{S}}
\def\T{\mathbb{T}}
\newcommand{\norm}[1]{\|#1\|}

\begin{document}
\title[Annihilating Skew-Selfadjoint Operators and Hilbert Complexes]
{Families of Annihilating Skew-Selfadjoint Operators\\
and their Connection to Hilbert Complexes}
\author{Dirk Pauly}
\author{Rainer Picard}
\address{Institut f\"ur Analysis, Technische Universit\"at Dresden, Germany}
\email[Dirk Pauly]{dirk.pauly@tu-dresden.de}
\email[Rainer Picard]{rainer.picard@tu-dresden.de}
\date{\today; {\it Corresponding Author}: Dirk Pauly, {\tt dirk.pauly@tu-dresden.de}}
\dedicatory{Dedicated to Alain Bossavit on the occasion of his 80th birthday}

\begin{abstract}
In this short note we show that Hilbert complexes are strongly related
to what we shall call annihilating sets of skew-selfadjoint operators.
This provides for a new perspective on the classical topic of Hilbert
complexes viewed as families of commuting normal operators.
\end{abstract}

\maketitle
\setcounter{tocdepth}{3} \tableofcontents{}


\section{Introduction}

The classical differential geometry topic of ``chain complexes''
has entered functional analysis as the topic of so-called ``Hilbert
complexes''. The purpose of this note is to inspect Hilbert complexes
from another functional analytical perspective linked to a four decades
old construction of the skew-selfadjoint extended Maxwell operator
\[
S_{\mathsf{Dir}}=\left(\begin{array}{cccc}
0 & \dv & 0 & 0\\
\mathring{\grad} & 0 & -\curl & 0\\
0 & \mathring{\curl} & 0 & \grad\\
0 & 0 & \mathring{\dv} & 0
\end{array}\right),
\]
see \cite{P1984b,P1985a}. For some pre-history and the scope of this
construction see ~\cite{PTW2017a}. We merely mention here that the
extended Maxwell system provides not only a deeper structural insight
into the system of Maxwell's equations but also shows a deep connection
to the Dirac equation. Indeed, the extended Maxwell operator has proven
to be useful in important applications such as boundary integral equations
in electrodynamics at low frequencies, see e.g. \cite{HS2022a} and
\cite{TV2007a,TV2007b,TV2007c,TV2007d}.

As it will turn out Hilbert complexes are intimately related to, indeed
generalized by, an abstract concept, which we shall refer to as annihilating
sets of skew-selfadjoint operators, which in turn is based on observations
made in connection with the extended Maxwell system. This is the subject
of the main part in Section \ref{sec:Annihilating-Skew-Selfadjoint-Op}.
Our final section, Section \ref{sec:Applications} serves to illustrate
the abstract setting by a number of more or less classical applications.

\section{\label{sec:Annihilating-Skew-Selfadjoint-Op}Annihilating Sets of
Skew-Selfadjoint Operators and Hilbert Complexes}

\subsection{Finite Sets of Annihilating Skew-Selfadjoint Operators}

We start with particular finite sets of commuting skew-selfadjoint
operators $S$ on a Hilbert space $H$, i.e., 
\[
S:\dom(S)\subset H\to H,\quad S^{*}=-S,
\]
which we shall refer to as a -- pair-wise -- annihilating set of skew-selfadjoint
operators.

\begin{defn}
A finite set $\mathS$ of skew-selfadjoint operators satisfying 
\[
\ran\left(S\right)\subseteq\ker\left(T\right),\quad S\not=T,\:S,T\in\mathS,
\]
is called an annihilating set of skew-selfadjoint operators. 
\end{defn}

For the rest of this section, let $\mathS$ be an annihilating set
of skew-selfadjoint operators.

\begin{rem}
Let $S,T\in\mathS$. 
\begin{enumerate}
\item 
$\mathS$ is a set of commuting operators. 
\item 
It holds $TS=0$ on $\dom\left(S\right)$ for $S\not=T$. 
\item 
The observation that $S=f_{S}\left(Q\right)$ for some suitable complex
valued functions $f_{S}:\R\to\mathbb{C}$ in the sense of a function
calculus associated with a single selfadjoint operator $Q$ provides
for other examples, which are not necessarily tridiagonal, cf.~Section \ref{sec:tridiag}.
\end{enumerate}
\end{rem}

As a consequence we have a straight-forward application of the projection
theorem the following generalized (orthogonal) Helmholtz decomposition.

\begin{thm}
${\displaystyle H=K\oplus_{H}\bigoplus_{S\in\mathS}\overline{\ran\left(S\right)}}$
with (generalised cohomology group) ${\displaystyle K\coloneqq\bigcap_{S\in\mathS}\ker\left(S\right)}$. 
\end{thm}

\subsection{A Special Case: Tridiagonal Operator Matrices}
\label{sec:tridiag}

We consider operator matrices of the form 
\begin{align*}
A\coloneqq\sum_{k=1}^{N}A_{k} & =\left(\begin{array}{ccccc}
0 & 0 & \cdots & \cdots & 0\\
\papia_{1} & 0 & \ddots &  & \vdots\\
0 & \ddots & \ddots & \ddots & \vdots\\
\vdots & \ddots & \ddots & 0 & 0\\
0 & \cdots & 0 & \papia_{N} & 0
\end{array}\right), & A_{k} & \coloneqq\left(\begin{array}{cccccccc}
0 & 0 & \cdots & \cdots & \cdots & \cdots & \cdots & 0\\
0 & \ddots & \ddots &  &  &  &  & \vdots\\
\vdots & \ddots & \ddots & 0 &  &  &  & \vdots\\
\vdots &  & 0 & 0 & 0 &  &  & \vdots\\
\vdots &  &  & \papia_{k} & 0 & 0 &  & \vdots\\
\vdots &  &  &  & 0 & \ddots & \ddots & \vdots\\
\vdots &  &  &  &  & \ddots & \ddots & 0\\
0 & \cdots & \cdots & \cdots & \cdots & \cdots & 0 & 0
\end{array}\right)
\end{align*}
closed and densely defined on a Cartesian product $H=H_{1}\times\cdots\times H_{N+1}$
of Hilbert spaces $H_{k}$. Then 
\begin{align*}
A^{*}=\sum_{k=1}^{N}A_{k}^{*} & =\left(\begin{array}{ccccc}
0 & \papia_{1}^{*} & 0 & \cdots & 0\\
0 & 0 & \ddots & \ddots & \vdots\\
\vdots & \ddots & \ddots & \ddots & 0\\
\vdots &  & \ddots & 0 & \papia_{N}^{*}\\
0 & \cdots & \cdots & 0 & 0
\end{array}\right), & A_{k}^{*} & =\left(\begin{array}{cccccccc}
0 & 0 & \cdots & \cdots & \cdots & \cdots & \cdots & 0\\
0 & \ddots & \ddots &  &  &  &  & \vdots\\
\vdots & \ddots & \ddots & 0 &  &  &  & \vdots\\
\vdots &  & 0 & 0 & \papia_{k}^{*} &  &  & \vdots\\
\vdots &  &  & 0 & 0 & 0 &  & \vdots\\
\vdots &  &  &  & 0 & \ddots & \ddots & \vdots\\
\vdots &  &  &  &  & \ddots & \ddots & 0\\
0 & \cdots & \cdots & \cdots & \cdots & \cdots & 0 & 0
\end{array}\right).
\end{align*}

Now let 
\[
S\coloneqq2\skw A=A-A^{*}=\sum_{k=1}^{N}S_{k},\qquad S_{k}\coloneqq2\skw A_{k}=A_{k}-A_{k}^{*}.
\]
Then
\begin{align*}
S & =\left(\begin{array}{ccccc}
0 & -\papia_{1}^{*} & 0 & \cdots & 0\\
\papia_{1} & 0 & \ddots & \ddots & \vdots\\
0 & \ddots & \ddots & \ddots & 0\\
\vdots & \ddots & \ddots & 0 & -\papia_{N}^{*}\\
0 & \cdots & 0 & \papia_{N} & 0
\end{array}\right), & S_{k} & =\left(\begin{array}{cccccccc}
0 & 0 & \cdots & \cdots & \cdots & \cdots & \cdots & 0\\
0 & \ddots & \ddots &  &  &  &  & \vdots\\
\vdots & \ddots & \ddots & 0 &  &  &  & \vdots\\
\vdots &  & 0 & 0 & -\papia_{k}^{*} &  &  & \vdots\\
\vdots &  &  & \papia_{k} & 0 & 0 &  & \vdots\\
\vdots &  &  &  & 0 & \ddots & \ddots & \vdots\\
\vdots &  &  &  &  & \ddots & \ddots & 0\\
0 & \cdots & \cdots & \cdots & \cdots & \cdots & 0 & 0
\end{array}\right)
\end{align*}
are skew-selfadjoint
and we have the following main result. 

\begin{thm}
\label{thm:hilcomset} 
$\mathS:=\left\{ S_{1},\dots,S_{N}\right\} $
is an annihilating set of skew-selfadjoint operators if and only if
$(\papia_{1},\dots,\papia_{N})$ is a Hilbert complex, i.e., for all
$k=1,\dots,N$ 
\[
\papia_{k}:\dom\left(\papia_{k}\right)\subseteq H_{k}\to H_{k+1}
\]
are closed and densely defined linear operators satisfying $\ran\left(\papia_{k}\right)\subseteq\ker\left(\papia_{k+1}\right)$,
$k=1,\ldots,N-1$. 
\end{thm}

\begin{proof}
The result follows by a straightforward calculation. See Appendix
\ref{app:proofs} for more details. 
\end{proof}

\begin{rem}
\label{rem:hilcomseq0} 
As $(S_{k}S_{\ell})^{*}\supset S_{\ell}^{*}S_{k}^{*}=S_{\ell}S_{k}$
we see that $\mathS=\left\{ S_{1},\dots,S_{N}\right\} $ is an annihilating
set of skew-selfadjoint operators if and only if $S_{k}S_{\ell}=0$
for all $1\leq k<\ell\leq N$
(if and only if $S_{\ell}S_{k}=0$ for all $1\leq k<\ell\leq N$). 
\end{rem}

\begin{rem}
\label{rem:hilcomseq1} 
Often a Hilbert complex $(\papia)\coloneqq\left(\papia_{1},\dots,\papia_{N}\right)$
is written as 
\[
H_{1}\xrightarrow{\papia_{1}}\cdots\xrightarrow{\papia_{k-1}}H_{k}\xrightarrow{\papia_{k}}H_{k+1}\xrightarrow{\papia_{k+1}}\cdots\xrightarrow{\papia_{N}}H_{N+1}.
\]
It is noteworthy that the Hilbert complex $\left(a\right)$ is equivalently
turned into the property 
\begin{align*}
\ran\left(A_{k}\right)&\subseteq\ker\left(A_{j}\right)\\
\intertext{or}
\ran\left(S_{k}\right)&\subseteq\ker\left(S_{j}\right),\quad
j\neq k,
\end{align*}
for all $j,k=1,\ldots,N$ where the sequential character of Hilbert
complexes seems to have disappeared. Theorem \ref{thm:hilcomset}
suggests to consider annihilating sets of skew-selfadjoint operators
as an appropriate generalization of Hilbert complexes. 
\end{rem}

\begin{rem}
Note that, if preferred, the set $\mathS$ may be considered as 
\begin{enumerate}
\item 
a set of homomorphisms by restriction of the elements to their respective
domains, i.e., 
\[
\mathS_{\mathsf{hom}}\coloneqq\left\{ \widetilde{S}_{1},\dots,\widetilde{S}_{N}\right\} ,
\]
where $\widetilde{S}_{k}\coloneqq S_{k}\iota_{\dom\left(S_{k}\right)}:\dom(S_{k})\to H$
are now bounded linear operators, 
\item 
a set of bounded isomorphisms by restriction of the elements to their
respective domains and orthogonal complements of their kernels (and
projections onto the ranges), i.e., 
\[
\mathS_{\mathsf{iso}}\coloneqq\left\{ \widehat{S}_{1},\dots,\widehat{S}_{N}\right\} ,
\]
where $\widehat{S}_{k}\coloneqq\iota_{\ran\left(S_{k}\right)}^{*}S_{k}\iota_{\dom\left(S_{k}\right)\cap\ker\left(S_{k}\right)^{\bot_{H}}}:\dom\left(S_{k}\right)\cap\ker\left(S_{k}\right)^{\bot_{H}}\to\ran(S_{k})$
are now bounded and bijective, 
\item 
a set of topological isomorphisms $\mathS_{\mathsf{iso}}$ if all
ranges $\ran(S_{k})$ are closed. Note that in this case we have $\ran(S_{k})=\ker\left(S_{k}\right)^{\bot_{H}}$
and that $\ran(S_{k})$ is closed if and only if $\ran(a_{k})$ is
closed. 
\end{enumerate}
\end{rem}

In the latter remark $\iota_{X}$ denotes the embedding of the subspace
$X$ into $H$. If $X$ is closed in $H$ the orthonormal projector
onto $X$ is given by $\pi_{X}:=\iota_{X}\iota_{X}^{*}:H\to H$.

\begin{rem}
\label{rem:domkerran} 
For consistency we set $a_{0}:=0$ and $a_{N+1}:=0$.
Note that 
\begin{align*}
\dom(S) & =\bigtimes_{k=1}^{N+1}\left(\dom(\papia_{k})\cap\dom(\papia_{k-1}^{*})\right)\\
 & =\dom(\papia_{1})\times\left(\dom(\papia_{2})\cap\dom(\papia_{1}^{*})\right)\times\cdots\times\left(\dom(\papia_{N})\cap\dom(\papia_{N-1}^{*})\right)\times\dom(\papia_{N}^{*})\intertext{and\:that\:by\:the\:complex\:property}\ker(S) & =\bigtimes_{k=1}^{N+1}\left(\ker(\papia_{k})\cap\ker(\papia_{k-1}^{*})\right)\\
 & =\ker(\papia_{1})\times\left(\ker(\papia_{2})\cap\ker(\papia_{1}^{*})\right)\times\cdots\times\left(\ker(\papia_{N})\cap\ker(\papia_{N-1}^{*})\right)\times\ker(\papia_{N}^{*}),\\
\ran(S) & =\bigtimes_{k=1}^{N+1}\left(\ran(\papia_{k-1})\oplus_{H_{k}}\ran(\papia_{k}^{*})\right)\\
 & =\ran(\papia_{1}^{*})\times\left(\ran(\papia_{1})\oplus_{H_{2}}\ran(\papia_{2}^{*})\right)\times\cdots\times\left(\ran(\papia_{N-1})\oplus_{H_{N}}\ran(\papia_{N}^{*})\right)\times\ran(\papia_{N}).
\end{align*}
In particular, the product of the cohomology groups $K_{k}\coloneqq\ker(\papia_{k})\cap\ker(\papia_{k-1}^{*})$
equals the kernel of $S$. 
\end{rem}

\begin{defn}
Recall Remark \ref{rem:hilcomseq1}. A Hilbert complex $(\papia)$
is called 
\begin{enumerate}
\item 
closed if all ranges $\ran(\papia_{k})$ are closed. 
\item 
compact if all embeddings $\dom(\papia_{k})\cap\dom(\papia_{k-1}^{*})\hookrightarrow H_{k}$
are compact.\textbf{ } 
\end{enumerate}
\end{defn}

\begin{thm}
\label{theo:Siffa} 
Recall Theorem \ref{thm:hilcomset}. 
Let $(\papia)$
be a Hilbert complex with associated annihilating set of skew-selfadjoint
operators $\mathS$. Then $(a)$ is 
\begin{enumerate}
\item 
closed if and only if $\ran(S)$ is closed. 
\item 
compact if and only if the embedding $\dom(S)\hookrightarrow H$ is
compact. 
\end{enumerate}
\end{thm}

\begin{proof}
Use Remark \ref{rem:domkerran} and orthogonality. 
\end{proof}

\begin{rem}
$S^{2}=\sum_{k=1}^{N}S_{k}^{2}$ is diagonal and may be considered
as generalised Laplacian acting on $H$. More precisely, 
\begin{align*}
-S^{2} & =\left(\begin{array}{ccccccc}
\papia_{1}^{*}\papia_{1} & 0 & \cdots & \cdots & \cdots & \cdots & 0\\
0 & \papia_{1}\papia_{1}^{*}+\papia_{2}^{*}\papia_{2} & 0 &  &  &  & \vdots\\
\vdots & 0 & \ddots & \ddots &  &  & \vdots\\
\vdots &  & \ddots & \papia_{k-1}\papia_{k-1}^{*}+\papia_{k}^{*}\papia_{k} & \ddots &  & \vdots\\
\vdots &  &  & \ddots & \ddots & 0 & \vdots\\
\vdots &  &  &  & 0 & \papia_{N-1}\papia_{N-1}^{*}+\papia_{N}^{*}\papia_{N} & 0\\
0 & \cdots & \cdots & \cdots & \cdots & 0 & \papia_{N}\papia_{N}^{*}
\end{array}\right),\\
-S_{k}^{2} & =\left(\begin{array}{cccccccc}
0 & 0 & \cdots & \cdots & \cdots & \cdots & \cdots & 0\\
0 & \ddots & \ddots &  &  &  &  & \vdots\\
\vdots & \ddots & 0 & 0 &  &  &  & \vdots\\
\vdots &  & 0 & \papia_{k}^{*}\papia_{k} & 0 &  &  & \vdots\\
\vdots &  &  & 0 & \papia_{k}\papia_{k}^{*} & 0 &  & \vdots\\
\vdots &  &  &  & 0 & 0 & \ddots & \vdots\\
\vdots &  &  &  &  & \ddots & \ddots & 0\\
0 & \cdots & \cdots & \cdots & \cdots & \cdots & 0 & 0
\end{array}\right)
\end{align*}
\end{rem}

\begin{rem}
By replacing the skew-selfadjoint operators with selfadjoint operators
the presented theory works literally as well. The only modifications
are 
\[
S\coloneqq2\sym A=A+A^{*}=\sum_{k=1}^{N}S_{k},\qquad S_{k}\coloneqq2\sym A_{k}=A_{k}+A_{k}^{*}
\]
resulting in 
\begin{align*}
S & =\left(\begin{array}{ccccc}
0 & \papia_{1}^{*} & 0 & \cdots & 0\\
\papia_{1} & 0 & \ddots & \ddots & \vdots\\
0 & \ddots & \ddots & \ddots & 0\\
\vdots & \ddots & \ddots & 0 & \papia_{N}^{*}\\
0 & \cdots & 0 & \papia_{N} & 0
\end{array}\right), & S_{k} & =\left(\begin{array}{cccccccc}
0 & 0 & \cdots & \cdots & \cdots & \cdots & \cdots & 0\\
0 & \ddots & \ddots &  &  &  &  & \vdots\\
\vdots & \ddots & \ddots & 0 &  &  &  & \vdots\\
\vdots &  & 0 & 0 & \papia_{k}^{*} &  &  & \vdots\\
\vdots &  &  & \papia_{k} & 0 & 0 &  & \vdots\\
\vdots &  &  &  & 0 & \ddots & \ddots & \vdots\\
\vdots &  &  &  &  & \ddots & \ddots & 0\\
0 & \cdots & \cdots & \cdots & \cdots & \cdots & 0 & 0
\end{array}\right).
\end{align*}
This is in a sense a matter of taste. We prefer, however, the skew-selfadjoint
setting, since it has the advantage of being closer to various applications,
such as Maxwell's and Dirac's equation (written in real form). We
note, in particular, that skew-selfadjointness is at the heart of
energy conservation. 
\end{rem}

\begin{rem}
\label{rem:toolboxgenmain}
By standard arguments of linear functional analysis we note the following results:
Let $S:\dom(S)\subset H\to H$ be skew-selfadjoint
such that $\dom(S)\hookrightarrow H$ is compact. Then
\begin{enumerate}
\item
the range $\ran(S)=\ran(\widehat{S})$ is closed, 
where $\widehat{S}=\iota_{\ran(S)}^{*}S\iota_{\ran(S)}$ .
\item
the inverse operator $\widehat{S}^{-1}$ is compact.
\item
the cohomology group $\ker(S)$ has finite dimension.
\item
the orthogonal Helmholtz-type decomposition 
$H=\ran(S)\oplus_{H}\ker(S)$ holds.
\item
there exists $c>0$ such that for all
$x\in\dom\big(\widehat{S}\big)
=\dom(S)\cap\ran(S)
=\dom(S)\cap\ker(S)^{\bot_{H}}$
the Friedrichs/Poincar\'e type inequality 
$\norm{x}_{H}\leq c\norm{S x}_{H}$
holds.
\item
$S$ and $\widehat{S}$ are Fredholm operators with index zero.
\end{enumerate}
\end{rem}

%
%

\section{\label{sec:Applications}Applications}

In this final section we give several examples of annihilating sets
of skew-selfadjoint operators, i.e., of Hilbert complexes, cf.~Theorem
\ref{thm:hilcomset}. All operators will be considered as closures
of unbounded linear operators densely defined on smooth and compactly
supported test fields. For example, $\mathring{\grad}$, $\mathring{\sym\Curl}_{\T}$,
and $\mathring{\dv\Div}_{\S}$ -- where the tiny circle on top of
an operator indicates the full Dirichlet boundary condition associated
to the respective differential operator -- are the closures of 
\begin{align*}
\mathring{\grad}^{\infty}:\mathring{C}^{\infty}(\Omega)\subset L^{2}(\Omega) & \to L^{2}(\Omega); & u & \mapsto\grad u,\\
\mathring{\sym\Curl}_{\T}^{\infty}:\mathring{C}_{\T}^{\infty}(\Omega)\subset L_{\T}^{2}(\Omega) & \to L_{\S}^{2}(\Omega); & M & \mapsto\sym\Curl M,\\
\mathring{\dv\Div}_{\S}^{\infty}:\mathring{C}_{\S}^{\infty}(\Omega)\subset L_{\S}^{2}(\Omega) & \to L^{2}(\Omega); & M & \mapsto\dv\Div M,
\end{align*}
where $\Omega\subset\R^{3}$ is an open set, $\mathring{C}^{\infty}(\Omega)$
denotes the space of smooth and compactly supported fields in $\Omega$,
and $\S$ and $\T$ indicate symmetric and deviatoric tensor fields,
respectively. The corresponding adjoints $-\dv$, $\Curl_{\S}$, and
$\Grad\grad$ are then given by 
\begin{align*}
-\dv:H(\dv,\Omega)\subset L^{2}(\Omega) & \to L^{2}(\Omega); & E & \mapsto-\dv E,\\
\Curl_{\S}:H_{\S}(\Curl,\Omega)\subset L_{\S}^{2}(\Omega) & \to L_{\T}^{2}(\Omega); & M & \mapsto\Curl M,\\
\Grad\grad:H^{2}(\Omega)\subset L^{2}(\Omega) & \to L_{\S}^{2}(\Omega); & u & \mapsto\Grad\grad u.
\end{align*}

\subsection{The Classical de Rham Complexes}

\subsubsection{De Rham Complex of Vector Fields}

Let $\Omega$ be an open set in $\R^{3}$ with boundary $\Gamma:=\partial\Omega$.
The most prominent example is the classical de Rham complex of vector
fields involving the classical operators of vector calculus $\grad$,
$\curl$, and $\dv$ with full Dirichlet or Neumann boundary conditions:
\[
S_{\mathsf{Dir}}=\left(\begin{array}{ccccc}
0 & \dv & 0 & 0\\
\mathring{\grad} & 0 & -\curl & 0\\
0 & \mathring{\curl} & 0 & \grad\\
0 & 0 & \mathring{\dv} & 0
\end{array}\right),\qquad S_{\mathsf{Neu}}=\left(\begin{array}{ccccc}
0 & \mathring{\dv} & 0 & 0\\
\grad & 0 & -\mathring{\curl} & 0\\
0 & \curl & 0 & \mathring{\grad}\\
0 & 0 & \dv & 0
\end{array}\right)
\]
Inhomogeneous and anisotropic coefficients and mixed boundary conditions
can also be considered: 
\begin{align}
S_{\mathsf{mix}}=\left(\begin{array}{cccc}
0 & \nu^{-1}\mathring{\dv}_{\Gamma_{\mathsf{1}}}\eps & 0 & 0\\
\mathring{\grad}_{\Gamma_{0}} & 0 & -\eps^{-1}\mathring{\curl}_{\Gamma_{1}} & 0\\
0 & \mu^{-1}\mathring{\curl}_{\Gamma_{0}} & 0 & \mathring{\grad}_{\Gamma_{1}}\\
0 & 0 & \kappa^{-1}\mathring{\dv}_{\Gamma_{0}}\mu & 0
\end{array}\right)\label{eq:derhammix}
\end{align}
Here the boundary $\Gamma$ is decomposed into two parts $\Gamma_{0}$
and $\Gamma_{1}$ where the Dirichlet and Neumann boundary condition
is imposed, respectively. Note that the de Rham off-diagonals are
skew-adjoint to each other.

\subsubsection{De Rham Complex of Differential Forms}

Let $\Omega$ be an $N$-dimensional Riemannian manifold, e.g., an
open set in $\R^{N}$. Another prominent example is the classical
de Rham complex of differential forms involving the exterior derivative
$\ed$ and its formal skew-adjoint the co-derivative $\cd=-\mathring{\ed}^{*}$,
$\ed=-\mathring{\cd}^{*}$ with full Dirichlet or Neumann boundary
conditions: 
\[
S_{\mathsf{Dir}}=\left(\begin{array}{ccccc}
0 & \cd & 0 & \cdots & 0\\
\mathring{\ed} & \ddots & \ddots & \ddots & \vdots\\
0 & \ddots & \ddots & \ddots & 0\\
\vdots & \ddots & \ddots & \ddots & \cd\\
0 & \cdots & 0 & \mathring{\ed} & 0
\end{array}\right),\qquad S_{\mathsf{Neu}}=\left(\begin{array}{ccccc}
0 & \mathring{\cd} & 0 & \cdots & 0\\
\ed & \ddots & \ddots & \ddots & \vdots\\
0 & \ddots & \ddots & \ddots & 0\\
\vdots & \ddots & \ddots & \ddots & \mathring{\cd}\\
0 & \cdots & 0 & \ed & 0
\end{array}\right)
\]
Again, the de Rham off-diagonals are skew-adjoint to each other. Note
that $S_{\mathsf{Dir}}$ and $S_{\mathsf{Neu}}$ are unitarily congruent
via transposition, permutation, sign change and Hodge-$*$-isomorphism.

\subsection{Other Complexes in Three Dimensions}


There are plenty of extensions and restrictions of the de Rham complex.
A nice overview and list of complexes is given in \cite{AH2021a},
from which we extract the following discussion. Their construction
is based on the BGG-resolution using copies of the de Rham complex.
For this section let $\Omega$ be an open set in $\R^{3}$.

\subsubsection{More De Rham Complexes}
\begin{itemize}
\item $\Grad\curl$ complex: 
\[
S_{\mathsf{Dir}}=\left(\begin{array}{cccccc}
0 & \dv & 0 & 0 & 0\\
\mathring{\grad} & 0 & \curl\Div_{\T} & 0 & 0\\
0 & \mathring{\Grad\curl} & 0 & -\dev\Curl & 0\\
0 & 0 & \mathring{\Curl}_{\T} & 0 & \Grad\\
0 & 0 & 0 & \mathring{\Div} & 0
\end{array}\right)
\]
\item $\curl\Div$ complex (formal dual of the $\Grad\curl$ complex): 
\[
S_{\mathsf{Dir}}=\left(\begin{array}{cccccc}
0 & \Div & 0 & 0 & 0\\
\mathring{\Grad} & 0 & -\Curl_{\T} & 0 & 0\\
0 & \mathring{\dev\Curl} & 0 & \Grad\curl & 0\\
0 & 0 & \mathring{\curl\Div}_{\T} & 0 & \grad\\
0 & 0 & 0 & \mathring{\dv} & 0
\end{array}\right)
\]
\item $\grad\dv$ complex (formally self-dual): 
\[
S_{\mathsf{Dir}}=\left(\begin{array}{cccccc}
0 & \dv & 0 & 0 & 0 & 0\\
\mathring{\grad} & 0 & -\curl & 0 & 0 & 0\\
0 & \mathring{\curl} & 0 & -\grad\dv & 0 & 0\\
0 & 0 & \mathring{\grad\dv} & 0 & -\curl & 0\\
0 & 0 & 0 & \mathring{\curl} & 0 & \grad\\
0 & 0 & 0 & 0 & \mathring{\dv} & 0
\end{array}\right)
\]
\end{itemize}

\subsubsection{Elasticity Complexes}
\begin{itemize}
\item Kr\"oner complex (formally self-dual): 
\[
S_{\mathsf{Dir}}=\left(\begin{array}{ccccc}
0 & \Div_{\S} & 0 & 0\\
\mathring{\sym\Grad} & 0 & -\Curl\top\Curl_{\S} & 0\\
0 & \mathring{\Curl\top\Curl}_{\S} & 0 & \sym\Grad\\
0 & 0 & \mathring{\Div}_{\S} & 0
\end{array}\right)
\]
Here $\top$ denotes the formal transpose. 
\item deviatoric Kr\"oner complex (formally self-dual): 
\[
S_{\mathsf{Dir}}=\left(\begin{array}{ccccc}
0 & \Div_{\S\T} & 0 & 0\\
\mathring{\dev\sym\Grad} & 0 & -\Curl\widetilde{\top}\Curl\widetilde{\top}\Curl_{\S\T} & 0\\
0 & \mathring{\Curl\widetilde{\top}\Curl\widetilde{\top}\Curl_{\S\T}} & 0 & \dev\sym\Grad\\
0 & 0 & \mathring{\Div}_{\S\T} & 0
\end{array}\right)
\]
Here\footnote{In $\mathbb{R}^{N}$ we have $\widetilde{\top}M:=M^{\top}-\frac{1}{N-1}(\tr M)\id$.}
$\widetilde{\top}M:=M^{\top}-\frac{1}{2}(\tr M)\id$. Note that $\widetilde{\top}\Curl_{\S}=\top\Curl_{\S}$
as $\tr\Curl_{\S}=0$.


\end{itemize}

\subsubsection{Biharmonic Complexes}
\begin{itemize}
\item first Hessian complex: 
\[
S_{\mathsf{Dir}}=\left(\begin{array}{ccccc}
0 & -\dv\Div_{\S} & 0 & 0\\
\mathring{\Grad\grad} & 0 & -\sym\Curl_{\T} & 0\\
0 & \mathring{\Curl}_{\S} & 0 & \dev\Grad\\
0 & 0 & \mathring{\Div}_{\T} & 0
\end{array}\right)
\]
\item second Hessian complex (formal dual of the first Hessian complex):
\[
S_{\mathsf{Dir}}=\left(\begin{array}{ccccc}
0 & \Div_{\T} & 0 & 0\\
\mathring{\dev\Grad} & 0 & -\Curl_{\S} & 0\\
0 & \mathring{\sym\Curl}_{\T} & 0 & -\Grad\grad\\
0 & 0 & \mathring{\dv\Div}_{\S} & 0
\end{array}\right)
\]
\item conformal Hessian complex (formally self-dual): 
\[
S_{\mathsf{Dir}}=\left(\begin{array}{ccccc}
0 & -\dv\Div_{\S\T} & 0 & 0\\
\mathring{\dev\Grad\grad} & 0 & -\sym\Curl_{\S\T} & 0\\
0 & \mathring{\sym\Curl}_{\S\T} & 0 & -\dev\Grad\grad\\
0 & 0 & \mathring{\dv\Div}_{\S\T} & 0
\end{array}\right)
\]
\end{itemize}

\subsection{Some Remarks}
\begin{rem}
Theorem \ref{thm:hilcomset} shows that all operators $S_{\cdots}=S_{\mathsf{Dir/Neu/mix}}$
are sums of annihilating skew-selfadjoint operators $S_{1},\dots,S_{N}$,
i.e., 
\[
S_{\cdots}=\sum_{k=1}^{N}S_{k}.
\]
In particular, we have for the Dirichlet de Rham complex 
\begin{align*}
S_{\mathsf{Dir}} & =\left(\begin{array}{ccccc}
0 & \dv & 0 & 0\\
\mathring{\grad} & 0 & -\curl & 0\\
0 & \mathring{\curl} & 0 & \grad\\
0 & 0 & \mathring{\dv} & 0
\end{array}\right)\\
 & =\left(\begin{array}{ccccc}
0 & \dv & 0 & 0\\
\mathring{\grad} & 0 & 0 & 0\\
0 & 0 & 0 & 0\\
0 & 0 & 0 & 0
\end{array}\right)+\left(\begin{array}{ccccc}
0 & 0 & 0 & 0\\
0 & 0 & -\curl & 0\\
0 & \mathring{\curl} & 0 & 0\\
0 & 0 & 0 & 0
\end{array}\right)+\left(\begin{array}{ccccc}
0 & 0 & 0 & 0\\
0 & 0 & 0 & 0\\
0 & 0 & 0 & \grad\\
0 & 0 & \mathring{\dv} & 0
\end{array}\right),\\
-S_{\mathsf{Dir}}^{2} & =\left(\begin{array}{ccccc}
-\dv\mathring{\grad} & 0 & 0 & 0\\
0 & -\mathring{\grad}\dv+\curl\mathring{\curl} & 0 & 0\\
0 & 0 & \mathring{\curl}\curl-\grad\mathring{\dv} & 0\\
0 & 0 & 0 & -\mathring{\dv}\grad
\end{array}\right)\\
 & =\left(\begin{array}{ccccc}
-\dv\mathring{\grad} & 0 & 0 & 0\\
0 & -\mathring{\grad}\dv & 0 & 0\\
0 & 0 & 0 & 0\\
0 & 0 & 0 & 0
\end{array}\right)+\left(\begin{array}{ccccc}
0 & 0 & 0 & 0\\
0 & \curl\mathring{\curl} & 0 & 0\\
0 & 0 & \mathring{\curl}\curl & 0\\
0 & 0 & 0 & 0
\end{array}\right)\\
 & \qquad+\left(\begin{array}{ccccc}
0 & 0 & 0 & 0\\
0 & 0 & 0 & 0\\
0 & 0 & -\grad\mathring{\dv} & 0\\
0 & 0 & 0 & -\mathring{\dv}\grad
\end{array}\right).
\end{align*}
\end{rem}

\begin{rem}
Recalling Theorem \ref{theo:Siffa} it has been shown in \cite{BPS2016a,BPS2019a,PS2022a}
that in case of the de Rham complexes the embeddings 
\[
\dom(S_{\cdots})\hookrightarrow H
\]
are compact, provided that $(\Omega,\Gamma_{0})$ is a bounded weak
Lipschitz pair, see \cite{W1974a,W1980a,P1984a,W1993a} and \cite{J1997a,FG1997a}
for the first results about the respective compact embeddings. For
corresponding results in case of elasticity and biharmonic complexes
and bounded strong Lipschitz pairs $(\Omega,\Gamma_{0})$ see \cite{PS2022b,PS2023a}
and \cite{PZ2020a,PZ2022a}. 
\end{rem}

\begin{rem}
There is no doubt that all the latter complexes may be generalised
to inhomogeneous and anisotropic coefficients and to mixed boundary
conditions, cf.~\eqref{eq:derhammix}. Moreover, the techniques of
\cite{PS2022a,PS2022b,PS2023a} (for the de Rham, Kr\"oner, and Hessian
complexes) can be extended to show that all the embeddings $\dom(S_{\cdots})\hookrightarrow H$
are compact for bounded strong Lipschitz pairs $(\Omega,\Gamma_{0})$. 
\end{rem}

\subsection{A Factorization Result}

An interesting consequence for annihilating sets of skew-selfadjoint
operators (not necessarily tridiagonal operator matrices as in the
special case) is the following factorization result: Let $T$ be a
strictly m-accretive operator commuting with $\mathS=\left\{ S_{1},\dots,S_{N}\right\} $.
Then with $S=\sum_{k=1}^{N}S_{k}$ we have 
\[
\left(T+S\right)=T^{1-N}T^{N-1}\Big(T+\sum_{k=1}^{N}S_{k}\Big)=T^{1-N}\prod_{k=1}^{N}\left(T+S_{k}\right)
\]
and conversely, 
\begin{align*}
T+S_{\ell} & =T^{N-1}\prod_{\ell\neq k=1}^{N}\left(T+S_{k}\right)^{-1}\left(T+S\right)=\prod_{\ell\neq k=1}^{N}\left(1+T^{-1}S_{k}\right)^{-1}\left(T+S\right).
\end{align*}
This is, in terms of inverses (solution operators) 
\begin{align*}
\left(T+S\right)^{-1}=T^{N-1}\prod_{k=1}^{N}\left(T+S_{k}\right)^{-1},\qquad\left(T+S_{\ell}\right)^{-1}=\prod_{\ell\neq k=1}^{N}\left(1+T^{-1}S_{k}\right)\left(T+S\right)^{-1}.
\end{align*}
An example of particular interest is given by $T=\partial_{t}$, the
case of evolutionary systems, which in a suitable setting, see e.g.
\cite{PTW2017a}, leads to 
\begin{align*}
(\partial_{t}+S) & =\partial_{t}^{1-N}\prod_{k=1}^{N}(\partial_{t}+S_{k}), & \partial_{t}+S_{\ell} & =\prod_{\ell\neq k=1}^{N}(1+\partial_{t}^{-1}S_{k})^{-1}(\partial_{t}+S),\\
(\partial_{t}+S)^{-1} & =\partial_{t}^{N-1}\prod_{k=1}^{N}(\partial_{t}+S_{k})^{-1}, & (\partial_{t}+S_{\ell})^{-1} & =\prod_{\ell\neq k=1}^{N}(1+\partial_{t}^{-1}S_{k})(\partial_{t}+S)^{-1}.
\end{align*}

\section*{Funding/Conflict of Interest Statement} 
The authors declare that they have no conflict of interest 
and that they obtained no funding.



\global\long\def\Sone{\left(\begin{array}{cccccccc}
0 & -\papia_{1}^{*} & 0 & 0\\
\papia_{1} & 0 & 0 & 0\\
0 & 0 & 0 & 0\\
0 & 0 & 0 & 0
\end{array}\right)}%
 
\global\long\def\Stwo{\left(\begin{array}{cccccccc}
0 & 0 & 0 & 0\\
0 & 0 & -\papia_{2}^{*} & 0\\
0 & \papia_{2} & 0 & 0\\
0 & 0 & 0 & 0
\end{array}\right)}%
 
\global\long\def\Sthree{\left(\begin{array}{cccccccc}
0 & 0 & 0 & 0\\
0 & 0 & 0 & 0\\
0 & 0 & 0 & -\papia_{3}^{*}\\
0 & 0 & \papia_{3} & 0
\end{array}\right)}%

\appendix

\section{Sketch of a Proof of Theorem \ref{thm:hilcomset}}

\label{app:proofs}

For simplicity and readability we look at the special case $N=3$
and consider $(\papia_{1},\papia_{2},\papia_{3})$. Then 
\begin{align*}
S_{1}S_{1} & =\Sone\Sone=\left(\begin{array}{cccccccc}
-\papia_{1}^{*}\papia_{1} & 0 & 0 & 0\\
0 & -\papia_{1}\papia_{1}^{*} & 0 & 0\\
0 & 0 & 0 & 0\\
0 & 0 & 0 & 0
\end{array}\right),\\
S_{1}S_{2} & =\Sone\Stwo=\left(\begin{array}{cccccccc}
0 & 0 & \papia_{1}^{*}\papia_{2}^{*} & 0\\
0 & 0 & 0 & 0\\
0 & 0 & 0 & 0\\
0 & 0 & 0 & 0
\end{array}\right),\\
S_{1}S_{3} & =\Sone\Sthree=\left(\begin{array}{cccccccc}
0 & 0 & 0 & 0\\
0 & 0 & 0 & 0\\
0 & 0 & 0 & 0\\
0 & 0 & 0 & 0
\end{array}\right),\\
S_{2}S_{1} & =\Stwo\Sone=\left(\begin{array}{cccccccc}
0 & 0 & 0 & 0\\
0 & 0 & 0 & 0\\
\papia_{2}\papia_{1} & 0 & 0 & 0\\
0 & 0 & 0 & 0
\end{array}\right),\\
S_{2}S_{2} & =\Stwo\Stwo=\left(\begin{array}{cccccccc}
0 & 0 & 0 & 0\\
0 & -\papia_{2}^{*}\papia_{2} & 0 & 0\\
0 & 0 & -\papia_{2}\papia_{2}^{*} & 0\\
0 & 0 & 0 & 0
\end{array}\right),\\
S_{2}S_{3} & =\Stwo\Sthree=\left(\begin{array}{cccccccc}
0 & 0 & 0 & 0\\
0 & 0 & 0 & \papia_{2}^{*}\papia_{3}^{*}\\
0 & 0 & 0 & 0\\
0 & 0 & 0 & 0
\end{array}\right),\\
S_{3}S_{1} & =\Sthree\Sone=\left(\begin{array}{cccccccc}
0 & 0 & 0 & 0\\
0 & 0 & 0 & 0\\
0 & 0 & 0 & 0\\
0 & 0 & 0 & 0
\end{array}\right),\\
S_{3}S_{2} & =\Sthree\Stwo=\left(\begin{array}{cccccccc}
0 & 0 & 0 & 0\\
0 & 0 & 0 & 0\\
0 & 0 & 0 & 0\\
0 & \papia_{3}\papia_{2} & 0 & 0
\end{array}\right),\\
S_{3}S_{3} & =\Sthree\Sthree=\left(\begin{array}{cccccccc}
0 & 0 & 0 & 0\\
0 & 0 & -\papia_{3}^{*}\papia_{3} & 0\\
0 & 0 & 0 & -\papia_{3}\papia_{3}^{*}
\end{array}\right).
\end{align*}
We read off that $(\papia_{1},\papia_{2},\papia_{3})$ is a Hilbert
complex if and only if $S_{k}S_{\ell}=0$ for all $k\neq\ell$. 


\vspace*{5mm}
\hrule 
\vspace*{3mm}


\end{document}